\documentclass[a4paper, 12pt]{article}

\usepackage[sort&compress]{natbib}
\bibpunct{(}{)}{;}{a}{}{,} 

\usepackage{amsthm, amsmath, amssymb, mathrsfs, multirow, url, subfigure}
\usepackage{graphicx} 
\usepackage{ifthen} 
\usepackage{amsfonts}
\usepackage[usenames]{color}
\usepackage{fullpage}

\theoremstyle{plain} 
\newtheorem{thm}{Theorem}

\newtheorem{lem}{Lemma}

\theoremstyle{definition}

\theoremstyle{remark}
\newtheorem{remark}{Remark}

\newcommand{\RR}{\mathbb{R}}

\newcommand{\E}{\mathsf{E}}

\newcommand{\prob}{\mathsf{P}}

\newcommand{\eps}{\varepsilon}
\renewcommand{\phi}{\varphi}

\newcommand{\ind}{\overset{\text{\tiny ind}}{\sim}}

\newcommand{\nm}{\mathsf{N}}
\newcommand{\be}{\mathsf{Beta}}
\newcommand{\ber}{\mathsf{Ber}}
\newcommand{\bin}{\mathsf{Bin}}

\renewcommand{\S}{\mathcal{S}}

\title{Asymptotically minimax empirical Bayes estimation of a sparse normal mean vector}
\author{
Ryan Martin \\
Department of Mathematics, Statistics, and Computer Science \\
University of Illinois at Chicago \\
\url{rgmartin@uic.edu} \\
\mbox{} \\
Stephen G. Walker \\
Department of Mathematics \\
University of Texas at Austin \\
\url{s.g.walker@math.utexas.edu}
}
\date{\today}

\begin{document}

\maketitle 


\begin{abstract}
For the important classical problem of inference on a sparse high-dimensional normal mean vector, we propose a novel empirical Bayes model that admits a posterior distribution with desirable properties under mild conditions.  In particular, our empirical Bayes posterior distribution concentrates on balls, centered at the true mean vector, with squared radius proportional to the minimax rate, and its posterior mean is an asymptotically minimax estimator.  We also show that, asymptotically, the support of our empirical Bayes posterior has roughly the same effective dimension as the true sparse mean vector.  Simulation from our empirical Bayes posterior is straightforward, and our numerical results demonstrate the quality of our method compared to others having similar large-sample properties.   

\medskip

\end{abstract}

\section{Introduction}
\label{S:intro}

High-dimensional problems, where the parameter is effectively lower-dimensional, are now commonplace in statistical applications.  Examples include variable selection in regression \citep{fan.lv.2010}, covariance matrix estimation \citep{lam.fan.2009, cai.zhang.zhou.2010, cai.zhou.2012}, large-scale multiple testing \citep{bcfg2010, cai.jin.2010}, and function estimation \citep{cai2012, johnstonesilverman2005}.  The canonical example, which we shall consider here, is that of estimating a sparse high-dimensional normal mean vector.  Let $X_1,\ldots,X_n$ be independent observations, with $X_i \sim \nm(\theta_i,1)$, $i=1,\ldots,n$, and the goal is to estimate the mean vector $\theta=(\theta_1,\ldots,\theta_n)$ under squared-error loss $\|\hat\theta - \theta\|^2$, where $\|\cdot\|$ is the usual $\ell_2$-norm on $\RR^n$ \citep[e.g.,][]{donohojohnstone1994b, abramovich2006, brown.greenshtein.2009, jiang.zhang.2009, castillo.vaart.2012, donoho1992}.  With only a single observation $X_i$ for each $\theta_i$, accurate estimation is not possible without some structure.  Assuming $\theta$ is sparse, in the sense that most of the $\theta_i$'s are zero, makes the effective dimension relatively small so that reasonably accurate estimation becomes possible. 

This normal means model is by now a classic one which has been widely studied from both a mathematical and applied point of view.  Despite the extent to which the many-normal-means model has been studied, it is still a practically important model in a variety of problems.  For example, the sparse normal mean model is the cornerstone for many modern Bayes and empirical Bayes multiple testing procedures, e.g., \citet{scottberger}, \citet{jincai2007}, \citet{bogdan.ghosh.tokdar.2008}, \citet{efron2008}, and \citet{mt-test}.  More recently, \citet{scott.FDRreg} have presented a novel use of the same classical model considered here but in the regression setting.  Clearly, research on this classical model is alive and well, and the results provided by our unique approach, namely, asymptotically minimax concentration rates and superior finite-sample performance compared to many existing methods, are useful contributions.

Recently, \citet{castillo.vaart.2012} have considered the performance of several Bayesian methods for this problem.  They focus on frequentist properties of a Bayesian posterior distribution, and the corresponding Bayes estimators, for priors with a two-groups structure.  In sparse estimation problems, a two-groups prior puts positive probability on $\theta$ vectors with some exact zero entries, so the marginal prior for each component is a mixture of a continuous distribution and a point-mass at zero.  \citet{castillo.vaart.2012} show that, for a suitably chosen two-groups prior, the posterior concentrates around the true signal at the asymptotically optimal minimax rate.  From this, concentration properties of posterior quantities, such as the posterior mean, can be derived.  An important message in their paper is that care is needed in choosing the prior for the non-zero $\theta$ entries.  In particular, they show that priors with too light tails, e.g., Gaussian, give sub-optimal concentration properties.  The results presented herein provide similar guidance, though our perspective is quite different.    

Here we take a novel empirical Bayes approach.  In particular, we present a hierarchical two-groups prior where, given a weight $\omega$, the $\theta_i$'s are modeled as independent, with $\theta_i = 0$ with probability $g_i(\omega)$, and $\theta_i \sim h_i(\theta \mid \omega)$ with probability $1-g_i(\omega)$, where the functions $g_i$ and $h_i$ depend on data $X_i$.  These functions are defined explicitly in Section~\ref{S:pseudo}.  To complete the hierarchy, $\omega$ is assigned a prior concentrated near 1.  We argue that the effect of the data-dependent prior is mitigated by preventing the posterior from tracking the data too closely.  This approach provides some new insights, which we compare with those coming from the fully Bayesian framework of \citet{castillo.vaart.2012}.  

In Section~\ref{S:theory} we present our theoretical framework.  First, we show that our empirical Bayes posterior concentrates, with probability~1, around the true mean vector at the optimal minimax rate (with respect to square error loss) for the assumed sparsity class.  Concentration rate theorems for empirical Bayes posteriors are relatively scarce in the literature, and our technique for handling the challenges that arise from data appearing in both the likelihood and prior might be useful in other problems; one possible extension is discussed briefly in Section~\ref{S:discuss}.  We then show that our empirical Bayes posterior mean is an asymptotically minimax estimator of $\theta$.  Finally, we show that, asymptotically, the support of our empirical Bayes posterior has, up to a logarithmic factor, the same effective dimension as the true sparse $\theta$.  An interesting observation is that the particular form of the prior on $\omega$ is the main catalyst for concentration of our empirical Bayes posterior.  

Section~\ref{S:numerical} describes computation of our empirical Bayes posterior mean via a straightforward Markov chain Monte Carlo.  Simulation results are presented to show that our empirical Bayes posterior mean generally outperforms those Bayesian and non-Bayesian competitors with comparable large-sample properties.  In particular, we compare our method with 
a two hard thresholding estimators \citep{donohojohnstone1994b}, Bayes and empirical Bayes estimators based on priors with a two-groups structure \citep{castillo.vaart.2012, johnstonesilverman2004}, and a new estimator based on the one-group Dirichlet--Laplace prior \citep{dunson.shrinkage}.  Our proposed empirical Bayes estimator is competitive in all cases considered here, and, in many cases, is strikingly better than the others.  Some concluding remarks are given in Section~\ref{S:discuss}.

\section{An empirical Bayes model}
\label{S:pseudo}

For the independent normal mean model, $X_i \sim \nm(\theta_i,1)$, $i=1,\ldots,n$, let $p_{\theta_i}(x_i)$ denote the density of $X_i$, and, for $x=(x_1,\ldots,x_n)$, let $p_\theta^n(x) = \prod_{i=1}^n p_{\theta_i}(x_i)$ denote the corresponding joint density of $X=(X_1,\ldots,X_n)$.  Define a data-dependent hierarchical prior $\Pi_X$ for $\theta=(\theta_1,\ldots,\theta_n)$ as follows.  Introduce a weight parameter $\omega \in (0,1)$, and take the joint prior distribution for $(\theta_1,\ldots,\theta_n,\omega)$, under $\Pi_X$, to have density proportional to 
\begin{equation}
\label{eq:ebprior}
\omega^{\alpha n - 1} \prod_{i=1}^n \Bigl\{ \omega e^{\frac12(1-\kappa)X_i^2} \delta_0(d\theta_i) + (1-\omega) \tfrac{1}{\sqrt{2\pi\sigma^2}} e^{-\frac12 \bigl[\frac{1-(1-\kappa)\sigma^2}{\sigma^2}\bigr](\theta_i-X_i)^2} \,d\theta_i \Bigr\}, 
\end{equation}
where $\alpha > 0$, $\kappa \in (0,1)$, and $\sigma^2 > 0$ are parameters to be discussed further in Sections~\ref{S:theory}--\ref{S:numerical}.  A representation of this as a genuine empirical Bayes plug-in prior is given in Section~\ref{SS:fraction}.  The dependence of the prior on $(\alpha,\kappa,\sigma^2)$ will not be reflected in our notation.  

Observe that if $\sigma^2 < (1-\kappa)^{-1}$, then the prior for $\theta_i$ is proper, a mixture of a point mass and a Gaussian centered at $X_i$.  When $\sigma^2 > (1-\kappa)^{-1}$, the prior is improper.  In any case, the posterior is proper, so this possible impropriety of the prior is not a concern.  In fact, $\sigma^2=(1-\kappa)^{-1}$ is a critical boundary, corresponding to an improper uniform prior for the non-zero $\theta_i$'s; see Section~\ref{SS:concentration}.  The term $\omega^{\alpha n - 1}$ in the joint density, which resembles a beta density, turns out to be critical to the success of our proposed method, both in theory and in implementation.  

Given data $X=(X_1,\ldots,X_n)$ from the normal mean model and the empirical Bayes prior distribution $\Pi_X$ for $\theta$, we could combine these to form an empirical Bayes posterior distribution via Bayes theorem.  That is, for a suitable set $A$ in the $\theta$-space, define the probability measure 
\[ Q_n(A) \equiv Q_{n,X}(A) \propto \int_A p_\theta^n(X) \, \Pi_X(d\theta). \]
We will investigate concentration properties of the empirical Bayes posterior in Section~\ref{S:theory}.  In particular, we show that the empirical Bayes posterior mean derived from $Q_n$ is an asymptotically minimax estimator of $\theta$.  

It might seem that our apparent double-use of the data---in the prior and in the likelihood---could lead to a posterior $Q_n$ that tracks the data too closely.  To see that this is not the case, note that if $|X_i|$ is large, then the prior probability for $\theta_i=0$, under $\Pi_X$, would be rather large.  Thus, the prior has an unexpected shrinkage effect, pushing $\theta_i$ corresponding to $X_i$ with large magnitude towards zero.  On the other hand, an $X_i$ with large magnitude shifts the prior on the non-zero part further from zero, effectively making the tails heavier, to accommodate large signals.  These two phenomena suggest that using data in both the prior and the likelihood will not result in a posterior that tracks data too closely.  In fact, our theoretical and numerical results demonstrate that the posterior is doing the right thing, namely, concentrating on the true $\theta$.

\section{Empirical Bayes posterior asymptotics}
\label{S:theory}

\subsection{A fractional likelihood perspective}
\label{SS:fraction}

To start, it will help to look at the proposed model from a different perspective.  For mathematical convenience, we shift our focus and rewrite the empirical Bayes posterior $Q_n$ using a fractional likelihood.  That is, we write $p_\theta^n(X) = p_\theta^n(X)^\kappa p_\theta^n(X)^{1-\kappa}$ and move the $1-\kappa$ fraction into the prior $\Pi_X$ defined above.  The effect of this is an alternative prior for $(\theta,\omega)$ of a very simple form: 
\begin{equation}
\label{eq:ebprior2}
\begin{split}
\theta_i \mid \omega & \ind \omega \delta_0 + (1-\omega) \nm(X_i, \sigma^2), \quad i=1,\ldots,n, \\
\omega & \sim \be(\alpha n, 1).
\end{split}
\end{equation}
To provide some further intuition for the prior \eqref{eq:ebprior} presented in Section~\ref{S:pseudo}, we may consider a data-free version of the prior in \eqref{eq:ebprior2}, where the $X_i$'s are replaced by hyperparameters $\mu_i$.  The marginal likelihood for $\mu=(\mu_1,\ldots,\mu_n)$, given $\omega$, is
\[ \prod_{i=1}^n \bigl\{ \omega \nm(X_i \mid 0,1) + (1-\omega) \nm(X_i \mid \mu_i, 1) \bigr\}. \]
and $X_i$ is clearly the maximum marginal likelihood estimate of $\mu_i$.  The use of plug-in estimates for mean hyperparameters was considered in \citet{babenko.belitser.2010} though in a slightly different context.  We get the empirical Bayes prior \eqref{eq:ebprior} by plugging in $X_i$ for $\mu_i$ and undoing the fractional likelihood.  

Within this alternative setup, we introduce independent binary latent variables $I_1,\ldots,I_n$, where $I_i=1$ if and only if $\theta_i=0$.  Then, given $\omega$, the indicators $I_1,\ldots,I_n$ are independent $\ber(\omega)$ variables.  These indicators  characterize the support of the vector $\theta$; in particular, $\sum_{i=1}^n (1-I_i)$ is the number of non-zero $\theta_i$ and is distributed as $\bin(n,1-\omega)$.  The beta prior for $\omega$ is concentrated near 1 for $n$ large, so the support size will tend to be small, consistent with the assumption of sparsity.  \citet{castillo.vaart.2012}, on the other hand, focus primarily on priors directly on the support size, though this kind of beta--binomial prior is considered in their Example~2.2.  We find that direct use of the weight $\omega$ is both theoretically and computationally convenient; see Remark~\ref{re:beta}.

Write this new version of the prior as $\tilde\Pi_X$, and express the posterior as 
\[ Q_n(A) \propto \int_A p_\theta^n(X)^\kappa \, \tilde\Pi_X(d\theta). \]
This version of the empirical Bayes posterior is particular amenable for our asymptotic analysis; see, also \citet{walker.hjort.2001}.  The use of pseudo-posteriors, where an inverse temperature parameter plays the role of $\kappa$, has been considered in the statistics and machine learning literature \citep[e.g.,][]{zhang2006, jiang.tanner.2008, dalalyan.tsybakov.2008}, but our context is different.

\subsection{Lower bound on the denominator}
\label{SS:denominator}

In the normal mean model, let $\theta^\star$ denote the true mean vector.  Assume that $\theta^\star$ is sparse in the sense that most of its entries are zero.  To make this more precise, let $\S^\star \subset \{1,2,\ldots,n\}$ denote the support of $\theta^\star$, i.e., $\theta_i^\star \neq 0$ if and only if $i \in \S^\star$.  Let $s_n = \#\S^\star$ be the cardinality of $\S^\star$, and say that $\theta^\star$ is $s_n$-sparse.  Then by sparse we mean that $s_n \to \infty$ but $s_n = o(n)$ as $n \to \infty$.  That is, although $\theta^\star$ is $n$-dimensional, its effective dimension is actually much smaller. 

Start by rewriting the empirical Bayes posterior $Q_n$ once more as 
\begin{equation}
\label{eq:pseudo}
Q_n(A) = \frac{\int_A \{ p_\theta^n(X) / p_{\theta^\star}^n(X) \}^\kappa \,\tilde\Pi_X(d\theta)}{\int_{\RR^n} \{ p_\theta^n(X) / p_{\theta^\star}^n(X) \}^\kappa \,\tilde\Pi_X(d\theta)}. 
\end{equation}
Our overall goal is to show that $Q_n$ concentrates its mass near $\theta^\star$ with $\prob_{\theta^\star}$-probability~1.  The strategy is to show that the denominator of $Q_n$ is not too small, and the numerator, for sets $A_n$ away from $\theta^\star$, is not too large.     

Our first result gives a bound on the denominator of $Q_n$, like that which obtains from the familiar Kullback--Leibler property \citep[e.g.,][]{schwartz1965, ggr1999, bsw1999, ggv2000, shen.wasserman.2001}.  This lower bound will be used in Section~\ref{SS:concentration} to derive vanishing upper bounds on the $Q_n$-probability assigned to complements of balls around $\theta^\star$.  But besides as a tool for proving other things, the following lemma suggests that our empirical Bayes-style prior is sufficiently concentrated around $\theta^\star$.  As \citet{castillo.vaart.2012} show, without suitable prior concentration, the desired posterior concentration is not possible.  Therefore, if we associate lower bounds on the denominator of $Q_n$ in \eqref{eq:pseudo} with adequate prior concentration, then Lemma~\ref{lem:denominator} says that our prior is sufficiently concentrated around $\theta^\star$.    

\begin{lem}
\label{lem:denominator}
Let $D_n$ be the denominator in \eqref{eq:pseudo}.  If $\theta^\star$ is $s_n$-sparse, then there exists $\eta \in \RR$, depending on $(\kappa,\alpha,\sigma^2)$, such that $D_n > \frac{\alpha}{1+\alpha} \exp\{\eta s_n - 2 s_n\log(n/s_n) + o(s_n)\}$ with $\prob_{\theta^\star}$-probability~1.  
\end{lem}

\begin{proof}
Write $D_n$ in terms of the conditional prior $(\theta_1,\ldots,\theta_n) \mid \omega \sim \tilde\Pi_{X,\omega}$ and the marginal prior $\omega \sim \tilde\pi$ for $\omega$ under $\tilde\Pi_X$.  That is, 
\[ D_n = \int_0^1 \int_{\RR^n} \Bigl\{ \frac{p_\theta^n(X)}{p_{\theta^\star}^n(X)} \Bigr\}^\kappa \, \tilde\Pi_{X,\omega}(d\theta) \,\tilde\pi(d\omega) = \int_0^1 \prod_{i=1}^n \int_{\RR} \Bigl\{ \frac{p_{\theta_i}(X_i)}{p_{\theta_i^\star}(X_i)}  \Bigr\}^\kappa \,\tilde\Pi_{X,\omega}(d\theta_i) \, \tilde\pi(d\omega). \]
For given $\omega$, the inner expectation involves an average over all configurations of the indicators $(I_1,\ldots,I_n)$ defined in Section~\ref{SS:fraction}.  This average is clearly larger than just the case where the indicators exactly match up with the support $\S^\star$ of $\theta^\star$, times the probability of that configuration.  That is, 
\begin{align*}
D_n > \int_0^1 \omega^{n-s_n}&(1-\omega)^{s_n} \, \tilde\pi(d\omega) \prod_{i \in \S^\star} \int_{\RR} e^{\frac{\kappa}{2}\{(X_i-\theta_i^\star)^2 - (X_i - \theta_i)^2\}} \frac{1}{\sqrt{2\pi\sigma^2}} e^{-\frac{1}{2\sigma^2}(X_i-\theta_i)^2} \,d\theta_i,  
\end{align*}
The term $\omega^{s_n-n}(1-\omega)^{s_n}$ corresponds to the probability for the configuration of $(I_1,\ldots,I_n)$ matching the support $\S^\star$.  The integral for $i \in \S^\star$ is the expectation of the normal density ratio for non-zero $\theta_i$ with respect to the $\nm(X_i,\sigma^2)$ prior.  Finally, the product over $i \not\in \S^\star$ disappears because $p_0(X_i) = p_{\theta_i^\star}(X_i)$ for $i \not\in \S^\star$.  To further bound this quantity, first pull out the terms $\exp\{\frac{\kappa}{2}(X_i-\theta_i^\star)^2\}$ in the latter integrand that do not depend on $\theta_i$.  Since, by the law of large numbers, $s_n^{-1} \sum_{i \in \S^\star} (X_i-\theta_i^\star)^2 \to 1$, as $n \to \infty$, with $\prob_{\theta^\star}$-probability~1, this part contributes a factor $\exp\{\frac{\kappa}{2} s_n + o(s_n)\}$ to the lower bound for $D_n$.  Next, 
\[ \int_{\RR} e^{-\frac{\kappa}{2}(X_i - \theta_i)^2} \frac{1}{\sqrt{2\pi\sigma^2}} e^{-\frac{1}{2\sigma^2}(\theta_i-X_i)^2} \,d\theta_i = \frac{1}{(1+\kappa\sigma^2)^{1/2}}. \]
So, the remaining product over $i \in \S^\star$ equals $(1 + \kappa\sigma^2)^{-s_n/2}$, and we can conclude that the entire product over $i \in \S^\star$ in the lower bound for $D_n$ is itself lower bounded by  
\[ \exp\Bigl[ \frac{s_n}{2} \{ \kappa - \log(1+\kappa\sigma^2) \} + o(s_n) \Bigr].  \]
It remains to bound the first integral over $\omega$.  Since $\pi(d\omega) = \alpha n \omega^{\alpha n-1}\,d\omega$, we have 
\begin{align}
\int_0^1 \omega^{n-s_n}(1-\omega)^{s_n} \,\tilde\pi(d\omega) & > \alpha n \int_0^{1-s_n/n} \omega^{n-s_n+\alpha n-1}(1-\omega)^{s_n} \,d\omega \notag \\
& > \Bigl( \frac{s_n}{n} \Bigr)^{s_n} \frac{\alpha n}{n-s_n+\alpha n} \Bigl(1 - \frac{s_n}{n} \Bigr)^{n-s_n + \alpha n} \notag \\
& > \frac{\alpha}{1+\alpha} \Bigl( \frac{s_n}{n} \Bigr)^{2s_n} \Bigl( 1 - \frac{s_n}{n} \Bigr)^{\alpha n}. \label{eq:rate}
\end{align}
The last inequality follows since $(1-b)^{1-b} > b^b$ for small $b > 0$.  Next, if we write 
\[ ( 1 - s_n/n )^{\alpha n} = \exp[-\alpha n\{-\log(1 - s_n/n)\}], \]
and use the approximation $-\log(1-x) = x + o(x)$, for $x \approx 0$, then we get a lower bound on the $\omega$-integral of the form:
\[ c \exp\bigl\{ -2s_n\log(n/s_n) - \alpha s_n + o(s_n) \bigr\}, \quad n \to \infty, \]
for $c=\alpha/(1+\alpha) > 0$.  Putting these pieces together, gives the lower bound 
\[ D_n > c \exp\Bigl[ \frac{s_n}{2} \{ \kappa - 2\alpha - \log(1+\kappa\sigma^2) + o(1) \} -2 s_n \log(n / s_n) \Bigr]. \]
Set $\eta = \frac{1}{2}\{\kappa - 2\alpha - \log(1+\kappa\sigma^2)\} \in \RR$ to complete the proof.  
\end{proof}

\subsection{Concentration}
\label{SS:concentration}

In the frequentist problem of estimating a $s_n$-sparse vector $\theta$ under squared $\ell_2$-error loss, it is known that the minimax rate is proportional to $\eps_n := s_n \log(n/s_n)$; see \citet{donoho1992}.  Following \citet{castillo.vaart.2012}, our goal here is to show that $Q_n$ concentrates asymptotically on $n$-balls, centered at $\theta^\star$, with square radius proportional to $\eps_n$.  More precisely, for a constant $M > 0$, let 
\[ A_{M\eps_n} = \{\theta \in \RR^n: \|\theta-\theta^\star\|^2 > M \eps_n \}; \]
then we will demonstrate that $Q_n(A_{M\eps_n}) \to 0$ with $\prob_{\theta^\star}$-probability~1.

The theorem below requires a restriction on $(\kappa,\sigma^2)$.  In particular, we require that, for some $\beta > 1$, $(\kappa,\sigma^2)$ reside in the feasible region
\begin{equation}
\label{eq:feasible}
R_\beta = \Bigl\{ (\kappa,\sigma^2): \frac{1}{\sigma^2 (1+\beta/\sigma^2)^{1/\beta}} - \frac{1}{\sigma^2+\beta} < \frac{\kappa[(1-\kappa)\beta-1]}{\beta-1} \Bigr\}.
\end{equation}
We are particularly interested in large $\beta$, so that $\kappa$ arbitrarily close to 1 can be included.  Figure~\ref{fig:feasible} displays a portion of the region $R_\beta$, for $\beta=200$.  The condition $\sigma^2=(1-\kappa)^{-1}$ discussed in Section~\ref{S:pseudo} defines the boundary of $R_\beta$, for large $\beta$ and $\kappa \approx 1$.     

\begin{figure}
\begin{center}
\scalebox{0.6}{\includegraphics{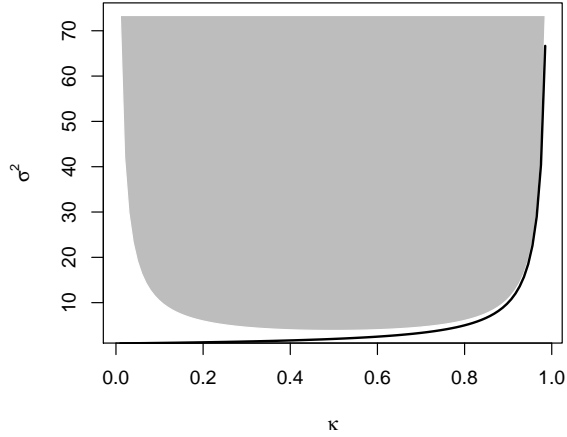}}
\end{center}
\caption{Portion of the feasible region $R_\beta$ in \eqref{eq:feasible}, with $\beta=200$, for $(\kappa,\sigma^2)$.  Solid black line corresponds to the curve $\sigma^2=(1-\kappa)^{-1}$.}
\label{fig:feasible}
\end{figure}

\begin{thm}
\label{thm:pseudo}
For any fixed $\beta > 1$, take $(\kappa,\sigma^2)$ in the feasible set $R_\beta$.  If $\theta^\star$ is $s_n$-sparse, then there exists $M > 0$ such that $Q_n(A_{M\eps_n}) \to 0$ with $\prob_{\theta^\star}$-probability~1.  
\end{thm}

\begin{proof}
Let $N_n$ be the numerator for $Q_n(A_{M\eps_n})$ in \eqref{eq:pseudo}, i.e.,   
\[ N_n = \int_0^1 \int_{A_{M\eps_n}} \prod_{i=1}^n \Bigl( \frac{p_{\theta_i}(X_i)}{p_{\theta_i^\star}(X_i)} \Bigr)^\kappa \,\tilde\Pi_{X_i,\omega}(d\theta_i) \, \tilde\pi(d\omega). \]
Taking expectation of $N_n$, with respect to $\prob_{\theta^\star}$, we get 
\[ \E_{\theta^\star}(N_n) = \int_0^1 \int_{A_{M\eps_n}} \prod_{i=1}^n \int_{\RR} \Bigl( \frac{p_{\theta_i}(x_i)}{p_{\theta_i^\star}(x_i)} \Bigr)^\kappa \,\tilde\Pi_{x_i,\omega}(d\theta_i) p_{\theta_i^\star}(x_i) \,dx_i \, \tilde\pi(d\omega). \]
Write $J_\omega(d\theta_i)$ for the measure defined in the $i$-th product term.  Split this into discrete and continuous pieces:
\begin{align*}
J_\omega(d\theta_i) & = \int \Bigl( \frac{p_{\theta_i}(x_i)}{p_{\theta_i^\star}(x_i)} \Bigr)^\kappa \,\tilde\Pi_{x_i,\omega}(d\theta_i) p_{\theta_i^\star}(x_i) \,dx_i \\
& = \omega \Bigl\{ \int \Bigl( \frac{p_{\theta_i}(x_i)}{p_{\theta_i^\star}(x_i)} \Bigr)^\kappa p_{\theta_i^\star}(x_i) \,dx_i \Bigr\} \, \delta_0(d\theta_i) \\
& \qquad\qquad\qquad\qquad + (1-\omega) \Bigl\{ \int \Bigl( \frac{p_{\theta_i}(x_i)}{p_{\theta_i^\star}(x_i)} \Bigr)^\kappa \frac{p_{\theta_i/\sigma}(x_i/\sigma)}{\sigma} p_{\theta_i^\star}(x_i) \,dx_i \Bigr\} \, d\theta_i.
\end{align*}
For clarity, we shall work with the discrete and continuous parts separately.

\emph{Discrete part}.  Using the Renyi divergence formula for normal distributions, the discrete term simplifies to $\omega \exp\{-\frac{\kappa(1-\kappa)}{2} (\theta_i-\theta_i^\star)^2\} \, \delta_0(d\theta_i)$.  

\emph{Continuous part}. An application of H\"older's inequality, with coefficients $\frac{\beta}{\beta-1}$ and $\beta$, whose reciprocals sum to one, gives 
\begin{align*}
\int \Bigl( \frac{p_{\theta_i}(x_i)}{p_{\theta_i^\star}(x_i)} & \Bigr)^\kappa \frac{p_{\theta_i/\sigma}(x_i/\sigma)}{\sigma}  p_{\theta_i^\star}(x_i) \,dx_i \\
& \leq \Bigl\{ \int \Bigl( \frac{p_{\theta_i}(x_i)}{p_{\theta_i^\star}(x_i)} \Bigr)^{\frac{\kappa\beta}{\beta-1}} p_{\theta_i^\star}(x_i) \,dx_i \Bigr\}^{\frac{\beta-1}{\beta}} \Bigl\{ \int \Bigl( \frac{p_{\theta_i/\sigma}(x_i/\sigma)}{\sigma} \Bigr)^\beta  p_{\theta_i^\star}(x_i) \,dx_i \Bigr\}^{\frac{1}{\beta}}. 
\end{align*}
For $(\kappa,\sigma^2) \in R_\beta$, we have $\frac{\kappa\beta}{\beta-1} < 1$.  Then the same Renyi divergence formula used above gives $\exp\{-\frac{\kappa}{2}\frac{\beta(1-\kappa)-1}{\beta-1} (\theta_i-\theta_i^\star)^2\}$.  The second term in the upper bound equals
\[ \frac{1}{\sqrt{2\pi \sigma^2}} \Bigl\{ \frac{\sigma}{(\sigma^2+\beta)^{1/2}} e^{-\frac{\beta}{2(\sigma^2+\beta)} (\theta_i-\theta_i^\star)^2} \Bigr\}^{1/\beta}. \]
After some tedious algebra, this can be rewritten as 
\[ \exp\Bigl\{\frac12 \Bigl(\frac{1}{\sigma^2 (1+\beta/\sigma^2)^{1/\beta}} - \frac{1}{\sigma^2+\beta} \Bigr)(\theta_i-\theta_i^\star)^2 \Bigr\} \nm(\theta_i \mid \theta_i^\star, \sigma^2 (1+\beta/\sigma^2)^{1/\beta}). \]
Combining the two terms in the upper bound, ignoring the normal density, gives 
\[ \exp\Bigl[ -\frac12\Bigl\{ \frac{\kappa[(1-\kappa)\beta-1]}{\beta-1} - \Bigl(\frac{1}{\sigma^2 (1+\beta/\sigma^2)^{1/\beta}} - \frac{1}{\sigma^2+\beta} \Bigr) \Bigr\} (\theta_i-\theta_i^\star)^2 \Bigr] . \]
For $(\kappa,\sigma^2)$ in the feasible region $R_\beta$ in \eqref{eq:feasible}, the coefficient on $(\theta_i-\theta_i^\star)^2$ in the exponential term above is negative.  

We can now find a constant $c > 0$, depending on $(\kappa,\sigma^2, \beta)$, such that 
\[ J_\omega(d\theta_i) \leq e^{-c(\theta_i-\theta_i^\star)^2} \{\omega \delta_0(d\theta_i) + (1-\omega) \nm(\theta_i \mid \theta_i^\star, \sigma^2 (1+\beta/\sigma^2)^{1/\beta}) \,d\theta_i\}. \]
Then $J_\omega^n(d\theta) := \prod_{i=1}^n J_\omega(d\theta_i)$ is upper bounded by $\exp\{-c\|\theta-\theta^\star\|^2\}$ times a probability measure in $\theta$ on $\RR^n$.  Therefore, by definition of $A_{M\eps_n}$, 
\[ \E_{\theta^\star}(N_n) = \int_0^1 \int_{A_{M\eps_n}} J_\omega^n(d\theta) \, \tilde\pi(d\omega) \leq e^{-c M \eps_n}. \]
Next, take $M$ such that $c M > 2$, and then take $K \in (2, cM)$.  Then Markov's inequality gives the upper bound
\[ \prob_{\theta^\star}(N_n > e^{-K\eps_n}) \leq L e^{-(cM-K)\eps_n}. \]
This upper bound has a finite sum over $n \geq 1$, so the Borel--Cantelli lemma gives that $N_n \leq e^{-K\eps_n}$, with $\prob_{\theta^\star}$-probability~1 for all large $n$.  Putting together this bound on $N_n$ and the one on $D_n$ from Lemma~\ref{lem:denominator}, we get 
\begin{equation}
\label{eq:bound}
\frac{N_n}{D_n} \leq \frac{1+\alpha}{\alpha} e^{-(K-2)\eps_n - \eta s_n + o(s_n)}. 
\end{equation}
Since $s_n = o(\eps_n)$, the exponent diverges to $-\infty$ regardless of the sign on $\eta$.  Therefore, $Q_n(A_{M\eps_n}) \to 0$ as $n \to \infty$ with $\prob_{\theta^\star}$-probability~1.
\end{proof}

\begin{remark}
\label{re:beta}
The $\eps_n$ concentration rate is driven primarily by the beta prior on the weight $\omega$.  In particular, it comes from the term $(s_n / n)^{2s_n}$ in the lower bound \eqref{eq:rate} in Lemma~\ref{lem:denominator}.  This means that the prior for $\theta$, given $\omega$, should be selected so that it does not interfere with the correct rate coming from the lower bound on the denominator of $Q_n$.    
\end{remark}

\begin{remark}
\label{re:interfere}
\citet{castillo.vaart.2012} show that the minimax concentration rate will not hold if the prior on non-zero $\theta$ has too light of tails, e.g., Gaussian.  A way to understand this point, from our perspective, is that the Gaussian conditional prior interferes with what the beta prior for the weight $\omega$ is doing.  As we have demonstrated, this does not necessarily mean that Gaussian is wrong, but that some adjustments should be made to prevent this interference.   
\end{remark}


\subsection{Asymptotic minimaxity of the posterior mean}

Since the empirical Bayes posterior concentrates around the right place and the right rate, it ought to produce an estimator of $\theta$ with good properties.  For this problem, perhaps the most natural choice of estimator is the empirical Bayes posterior mean, 
\[ \hat\theta_n = \int \theta \, Q_n(d\theta) \]
Next we show that $\hat\theta_n$ is a minimax estimator if $\theta^\star$ is $s_n$-sparse.  

\begin{thm}
\label{thm:minimax}
Take $(\kappa,\sigma^2)$ as in Theorem~\ref{thm:pseudo}.  If $\theta^\star$ is $s_n$-sparse, then there exists a universal constant $M' > 0$ such that $\E_{\theta^\star}\|\hat\theta_n - \theta^\star\|^2 \leq M' \eps_n$ for all large $n$.  
\end{thm}

\begin{proof}
Start by considering the quantity $\int \|\theta-\theta^\star\|^2 \, Q_n(d\theta)$.  Split this integral into two via the partition $\RR^n = A_{M\eps_n} \cup A_{M\eps_n}^c$ for $M$ as in Theorem~\ref{thm:pseudo}.  On $A_{M\eps_n}^c$, $\|\theta-\theta^\star\|^2$ is bounded above by $M\eps_n$, and $Q_n(A_{M\eps_n}^c) \leq 1$ trivially.  So, we immediately get 
\[ \int_{A_{M\eps_n}^c} \|\theta-\theta^\star\|^2 \, Q_n(d\theta) \leq M \eps_n. \]
For the integration over $A_{M\eps_n}$, we again look at the numerator and denominator of $Q_n$ separately, as in the previous subsection.  The denominator has the same lower bound as in Lemma~\ref{lem:denominator}.  Take $n$ large enough that, with $\prob_{\theta^\star}$-probability~1, the lower bound in the lemma holds; then the expectation of the ratio can be bounded by upper bounding the expectation of the numerator, together with the lower bound on the denominator.  Expectation of the numerator, with respect to $\prob_{\theta^\star}^n$, proceeds just like in the proof of Theorem~\ref{thm:pseudo}.  This time, we get
\[ \int_0^1 \int_{A_{M\eps_n}} \|\theta-\theta^\star\|^2 J_\omega^n(d\theta) \, \tilde\pi(d\omega), \]
where $J_\omega^n(d\theta)$ is $\exp(-c\|\theta-\theta^\star\|^2)$ times a probability measure for $\theta$ in $\RR^n$, just as in the proof of Theorem~\ref{thm:pseudo}.  Since the function $x \mapsto x e^{-c x}$ is monotonically decreasing for large enough $x$, we can see that, for large $n$, $\|\theta-\theta^\star\|^2 \exp(-c\|\theta-\theta^\star\|^2) < M\eps_n \exp(-cM\eps_n)$ on $A_{M\eps_n}$.  Therefore, the expectation is eventually bounded by $M\eps_n \exp(-cM\eps_n)$.  Combining this with the lemma's lower bound, we can find $\nu > 0$ such that, for large $n$, 
\[ \E_{\theta^\star} \int_{A_{M\eps_n}} \|\theta-\theta^\star\|^2 \, Q_n(d\theta) \leq M \eps_n e^{-\nu \eps_n}. \]
But $\|\hat\theta_n - \theta^\star\|^2 \leq \int \|\theta-\theta^\star\|^2 \, Q_n(d\theta)$ by Jensen's inequality, so $\E_{\theta^\star}\|\hat\theta-\theta^\star\|^2 \leq M\eps_n(1 + e^{-\nu\eps_n})$.  Take $M'=2M$ to complete the proof.
\end{proof}

\subsection{Effective posterior dimension}
\label{SS:dimension}

Besides posterior concentration around $\theta^\star$ at the minimax rate, it is desirable if the majority of the posterior mass is concentrated in a roughly $s_n$-dimensional subspace of $\RR^n$, where it is presumed that $\theta^\star$ resides.  \citet{castillo.vaart.2012} show that their fully Bayes posteriors have effective dimension proportional to $s_n$.  An interesting question, therefore, is if a similar result obtains for our empirical Bayes posterior.  In this section we show that, under the conditions of Theorems~\ref{thm:pseudo}--\ref{thm:minimax}, the posterior distribution for $1-\omega$ puts vanishingly small mass above $s_n n^{-1}$ (up to a logarithmic factor), so that $\omega$ tends to concentrate around $1-s_n n^{-1}$.  That this provides some information about the effective dimension of the posterior can be seen from the following expression:
\begin{equation}
\label{eq:alpha}
\E(\omega \mid X) = \frac{\alpha}{\alpha + 1 + n^{-1}} + \frac{1}{\alpha + 1 + n^{-1}} \frac{\E(D_\theta \mid X)}{n}, 
\end{equation}
where $D_\theta = \#\{i: \theta_i = 0\}$; this fact derives from the full conditionals in Section~\ref{SS:algorithm} below.  So, if $\alpha$ is not too large, and $\omega$ concentrates around $1-s_n n^{-1}$, then $D_\theta$ concentrates around $n-s_n$.  Therefore, the posterior distribution for $\theta$ must reside on a space with effective dimension proportional to $s_n$.    

\begin{thm}
\label{thm:dimension}
Let $\delta_n = K \eps_n n^{-1}$, where $\eps_n = s_n \log(n/s_n)$ as before, and $K > 0$ is a suitably large constant.  Then, under the conditions of Theorem~\ref{thm:pseudo}, 
\[ \E_{\theta^\star}\{\prob(1-\omega > \delta_n \mid X)\} \to 0 \quad \text{as $n \to \infty$}. \]
\end{thm}

\begin{proof}
Write the numerator of $\prob(1-\omega > \delta_n \mid X)$ as 
\[ N_n = \int_0^{1-\delta_n} \int_{\RR^n} \prod_{i=1}^n \Bigl( \frac{p_{\theta_i}(X_i)}{p_{\theta_i^\star}(X_i)} \Bigr)^\kappa \, \tilde \Pi_{X_i,\omega}(d\theta_i) \, \tilde\pi(d\omega). \]
This is similar to the first display in the proof of Theorem~\ref{thm:pseudo}.  Just as in that proof, we get the following bound on the expectation: 
\[ \E_{\theta^\star}(N_n) \leq \int_0^{1-\delta_n} \prod_{i=1}^n \int_{\RR} e^{-\frac{c}{2}(\theta_i-\theta_i^\star)^2} \{\omega \delta_0(d\theta_i) + (1-\omega) \nm(d\theta_i \mid \theta_i^\star, v)\} \, \tilde\pi(d\omega), \]
where $c$ is a positive constant and $v=v(\sigma^2,\beta)$ is a variance term that depends on the particular $\sigma^2$ and $\beta$ values.  Each integral in the inside product is bounded above by 1, so we get 
\[ \E_{\theta^\star}(N_n) \leq \int_0^{1-\delta_n} \tilde \pi(d\omega) = \alpha n \int_0^{1-\delta_n} \omega^{\alpha n - 1} \,d\omega \leq e^{-\alpha n \delta_n} = e^{-K\alpha \eps_n}. \]
From Lemma~\ref{lem:denominator}, we have that the denominator of $\prob(1-\omega > \delta_n \mid X)$ is lower bounded by $\exp\{-2\eps_n + O(s_n)\}$ with probability~1 for large $n$.  So, for large $n$, we get  
\[ \E_{\theta^\star}\{\prob(1-\omega > \delta_n \mid X)\} \leq \E_{\theta^\star}(N_n) e^{2\eps_n + O(s_n)} \leq e^{-(K\alpha - 2)\eps_n + O(s_n)}. \]
If we pick $K$ such that $K\alpha > 2$, then the fact that $s_n = o(\eps_n)$ implies that this upper bound approaches zero as $n \to \infty$, proving the claim.  
\end{proof}

Since the logarithmic term $\log(n/s_n)$ is small, the practical implication of this result is that the posterior distribution of $\omega$ concentrates around $1-s_n n^{-1}$.  The simulation results displayed in Figure~\ref{fig:wpost} below confirm this.  


\section{Numerical results}
\label{S:numerical}

\subsection{Computational considerations}
\label{SS:algorithm}

Computation of the empirical Bayes posterior mean can be carried out via a simple Gibbs sampler for $\omega$ and $\theta=(\theta_1,\ldots,\theta_n)$ based on the full conditionals: 
\begin{subequations}
\label{eq:full.cond}
\begin{align}
\theta_i \mid \omega, X & \ind \begin{cases} \delta_0 & \text{with prob.~$\propto \omega e^{-\frac{\kappa}{2}X_i^2}$} \\ \nm(X_i, \frac{\sigma^2}{1+\kappa \sigma^2}) & \text{with prob.~$\propto \frac{1-\omega}{\sqrt{1+\kappa\sigma^2}}$}, \end{cases} \quad i=1,\ldots,n \label{eq:full.cond.1} \\
\omega \mid \theta, X & \sim {\sf Beta}(\alpha n + D_\theta, 1+n-D_\theta), \label{eq:full.cond.2}
\end{align}
\end{subequations}
where $D_\theta = \#\{i: \theta_i = 0\}$.  That is, first sample from the $\theta \mid \omega$ conditional posterior in \eqref{eq:full.cond.1}, then from the $\omega \mid \theta$ conditional posterior in \eqref{eq:full.cond.2}.  Repeat this process to obtain a sample from the full posterior.  R code for this Gibbs sampling procedure is available at \url{www.math.uic.edu/~rgmartin}.  Once the posterior sample is available, the empirical Bayes estimator $\hat\theta$, the posterior mean, is obtained by computing a coordinate-wise average of the posterior $\theta$ samples.  Besides the posterior mean, many other quantities of interest can be calculated.  For example, inclusion probabilities, $\prob(\theta_i \neq 0 \mid X)$, $i=1,\ldots,n$, can be easily calculated.  Also, in a function estimation problem, where $\theta_1,\ldots,\theta_n$ are coefficients attached to the fixed basis functions, the posterior samples of the unknown functions are readily available.  

Theory and experience suggest that good numerical results are obtained for large $\kappa$ and large $\sigma^2$.  Throughout, we use $\kappa=0.99$ and $\sigma^2=(1-0.99)^{-1}=100$, on the boundary of the feasible region.  For $\alpha$, \eqref{eq:alpha} suggests that relatively small values of $\alpha$ are appropriate, so that the $\omega$ posterior can learn from $X$ through $D_\theta$.  We have found that choosing $\alpha$ to be decreasing with $n$ is a reasonable choice.  (This has no consequence on the results in Theorems~\ref{thm:pseudo}--\ref{thm:dimension}.)  In particular, in the three examples below, with $n=200, 500, 1000$ we take $\alpha=0.25, 0.10, 0.05$, respectively.  Alternatively, one could use the data to choose $\alpha$.  For example, a method-of-moments estimator of $\alpha$ can be obtained as follows. First, estimate $D=D_\theta$ via universal hard thresholding, i.e., $\hat D$ equals the number of $X_i$ such that $|X_i| \leq (2 \log n)^{1/2}$.  Under the assumed prior, $D$ has a beta--binomial distribution, with expectation $n^2\alpha / (n \alpha + 1)$.  If we set this expectation equal to $\hat D$, then solving for $\alpha$ gives a method-of-moments estimator, in particular, $\hat\alpha = \hat D \{n(n-\hat D)\}^{-1}$.  In our examples below, we use the $n$-dependent but data-free choices of $\alpha$ mentioned above.

\subsection{Simulation studies}
\label{SS:sims}

For illustration, we first reproduce a simulation study presented in \citet{dunson.shrinkage}.  In particular, we take samples $X=(X_1,\ldots,X_n)$ of dimension $n=200$ from the normal mean model $X_i \sim \nm(\theta_i^\star,1)$.  Recall the sparsity level $s_n$ is the number of non-zero $\theta_i^\star$'s.  In this case, we consider $s_n=10, 20, 40$, and the signals are fixed at values $A=7,8$.  Table~\ref{table:sim1} displays estimates of the mean squared error obtained from 100 replications of $X$.  In addition to the proposed empirical Bayes posterior mean estimator (EBM), based on $\kappa=0.99$, $\sigma^2=100$, and $\alpha=0.25$, the methods being compared are a Dirichlet--Laplace estimator (DL) of \citet{dunson.shrinkage}, an empirical Bayes median estimator (EBMed) of \citet{johnstonesilverman2004}, and a fully Bayes posterior median estimator (PMed1) of \citet{castillo.vaart.2012}.  A few other methods have been considered in the literature recently, and some comments on why they are omitted from comparison here are given in Remark~\ref{re:koenker} below.  Here, we find that our proposed empirical Bayes estimator is the top performer across all these settings.    

\begin{table}
\begin{center}
\begin{tabular}{cccccccccc}
\hline
$s_n$ & & \multicolumn{2}{c}{10} & & \multicolumn{2}{c}{20} & & \multicolumn{2}{c}{40} \\
\cline{3-4} \cline{6-7} \cline{9-10} 
$A$ & & 7 & 8 & & 7 & 8 & & 7 & 8 \\
\hline
DL$_{1/n}$ & & 16 & 14 & & 33 & 31 & & 66 & 60  \\
EBMed & & 26 & 26 & & 57 & 56 & & 119 & 119 \\
PMed1 & & 23 & 22 & & 49 & 48 & & 102 & 102 \\
\\
\emph{EBM} & & {\bf 13} & {\bf 13} & & {\bf 25} & {\bf 25} & & {\bf 47} & {\bf 48} \\
\hline
\end{tabular}
\end{center}
\caption{Mean square errors, based on 100 replications, sampling $X$ of dimension $n=200$.  First three rows are from \citet{dunson.shrinkage}; last row corresponds to the proposed empirical Bayes posterior mean.  Boldface font indicates the column winner.}
\label{table:sim1}
\end{table}

Consider a single sample $X$ under the simulation setting described above, with $n=200$, where the first $s_n=10$ entries in $\theta^\star$ equal $A=7$, and the remaining entries are zero.  For the given $X$, the Gibbs sampler is run to obtain a sample from our empirical Bayes posterior distribution of $\theta$.  In Figure~\ref{fig:inclpr} we plot the posterior inclusion probability $\prob(\theta_i \neq 0 \mid X)$ as a function of the indices $i=1,\ldots,n$.  It is evident that the empirical Bayes posterior is able to clearly identify the correct model.  

\begin{figure}
\begin{center}
\scalebox{0.6}{\includegraphics{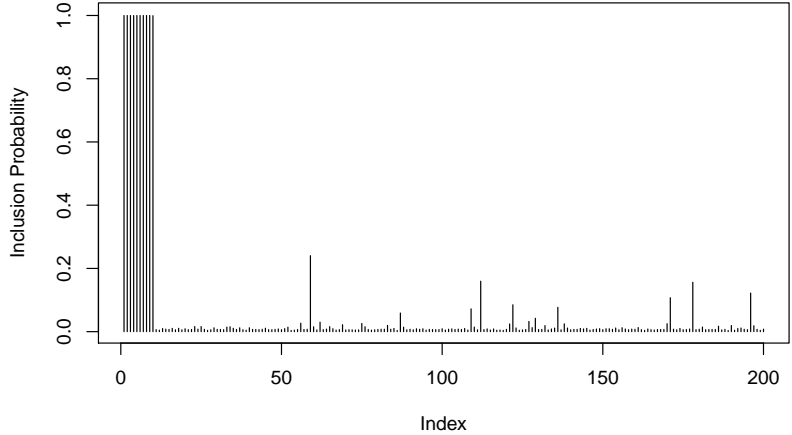}}
\end{center}
\caption{Plot of the empirical Bayes posterior inclusion probability $\prob(\theta_i \neq 0 \mid X)$ for $i=1,\ldots,n$.  Here $n=200$, $s_n=10$, and $\theta_1^\star=\cdots=\theta_{10}^\star=7$.}
\label{fig:inclpr}
\end{figure}

As a second example, we reproduce a simulation study presented in \citet{castillo.vaart.2012}.  In this case, we look at $n=500$, $s_n=25,50,100$, and signals fixed at $A=3,4,5$.  Table~\ref{table:sim2} displays estimates of the mean squared error based on 100 replications.  This time, the methods are two fully Bayes posterior mean estimates (PM1 and PM2), two fully Bayes component-wise posterior medians (PMed1 and PMed2), \citet{johnstonesilverman2004} empirical Bayes mean (EBM) and median (EBMed), and hard thresholding (HT) and hard thresholding oracle (HTO) rules.  Our proposed empirical Bayes estimator, based on $\alpha=0.10$, is competitive when $A=4$, and clearly dominates when $A=5$, just like in the previous illustration.  Interestingly, the empirical Bayes estimators are the better performers overall in this case.  


One rather unusual observation is that some of the methods have, for given $s_n$, a mean square error increasing in the signal size $A$.  We find this behavior to be counterintuitive, since it should be easier to detect stronger signals.  The two thresholding estimators have decreasing mean square error as $A$ increases, as does our proposed estimator.

\begin{table}
\begin{center}
\begin{tabular}{ccccccccccccc}
\hline
$s_n$ & & \multicolumn{3}{c}{25} & & \multicolumn{3}{c}{50} & & \multicolumn{3}{c}{100} \\
\cline{3-5} \cline{7-9} \cline{11-13} 
$A$ & & 3 & 4 & 5 & & 3 & 4 & 5 & & 3 & 4 & 5 \\
\hline
PM1 & & 111 & 96 & 94 & & 176 & 165 & 154 & & 267 & 302 & 307 \\
PM2 & & 106 & 92 & 82 & & 169 & 165 & 152 & & 269 & 280 & 274 \\
EBM & & {\bf 103} & 96 & 93 & & 166 & 177 & 174 & & 271 & 312 & 319 \\
\\
PMed1 & & 129 & 83 & 73 & & 205 & 149 & 130 & & 255 & 279 & 283 \\
PMed2 & & 125 & 86 & 68 & & 187 & {\bf 148} & 129 & & 273 & 254 & 245 \\
EBMed & & 110 & {\bf 81} & 72 & & {\bf 162} & {\bf 148} & 142 & & {\bf 255} & 294 & 300 \\
\\
HT & & 175 & 142 & 70 & & 339 & 284 & 135 & & 676 & 564 & 252 \\
HTO & & 136 & 92 & 84 & & 206 & 159 & 139 & & 306 & 261 & 245 \\
\\
\emph{EBM} & & 139 & 99 & {\bf 54} & & 237 & 159 & {\bf 89} & & 386 & {\bf 245} & {\bf 152} \\
\hline
\end{tabular}
\end{center}
\caption{Mean square errors, based on 100 replications, sampling $X$ of dimension $n=500$.  First eight rows are from \citet{castillo.vaart.2012}; last row corresponds to the proposed empirical Bayes posterior mean.  Boldface font indicates the column winner.}
\label{table:sim2}
\end{table}

To follow up on the mean square error results in Table~\ref{table:sim2}, we also display the posterior distribution of $\omega$ for two separate runs.  As indicated from Theorem~\ref{thm:dimension}, the posterior distribution for $\omega$ should concentrate around $1-s_n n^{-1}$.  For both cases in Figure~\ref{fig:wpost}, the posterior is concentrated exactly where we expect that it would be.

\begin{figure}
\begin{center}
\subfigure[$s_n=50$, so $1-\frac{s_n}{n}=0.90$]{\scalebox{0.55}{\includegraphics{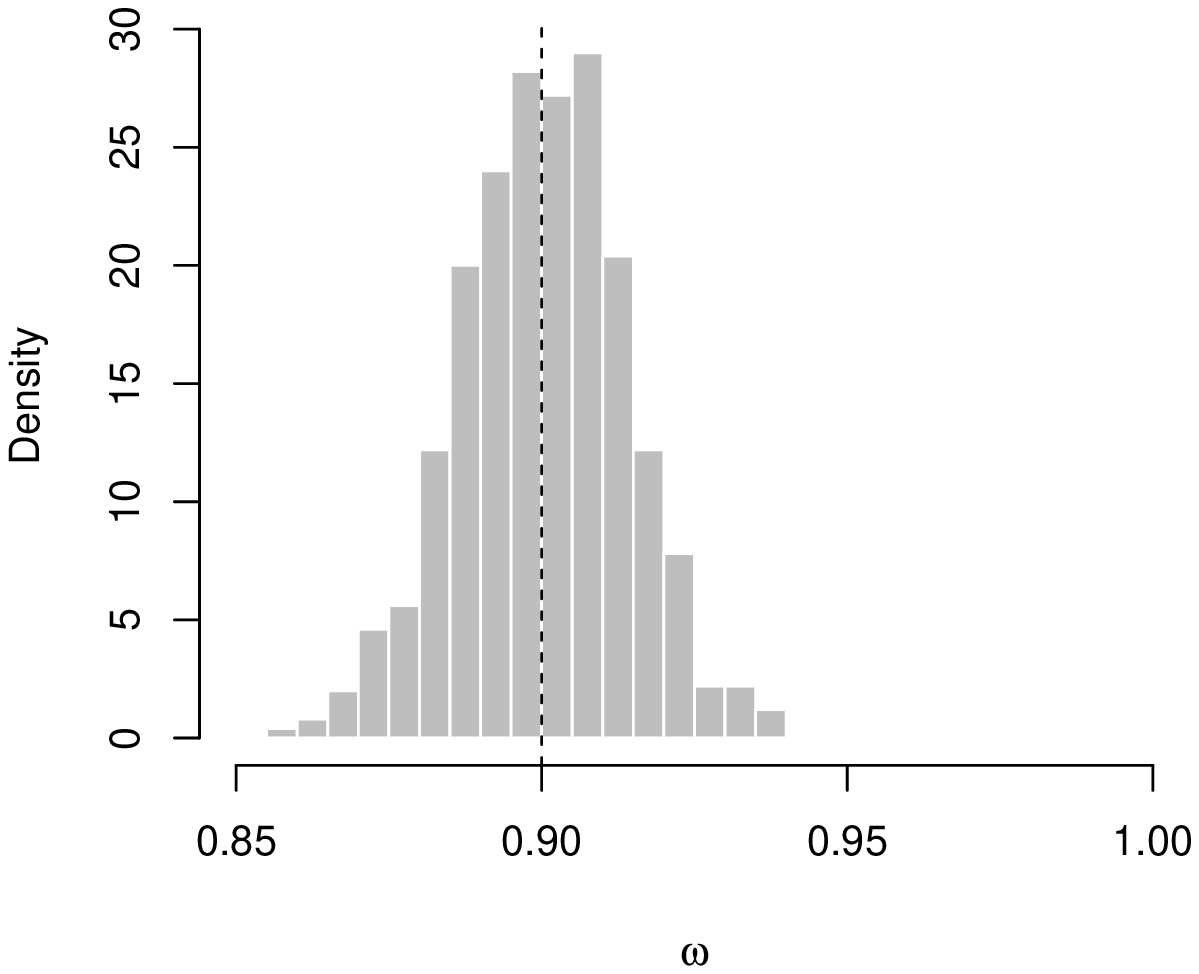}}}
\subfigure[$s_n=25$, so $1-\frac{s_n}{n}=0.95$]{\scalebox{0.55}{\includegraphics{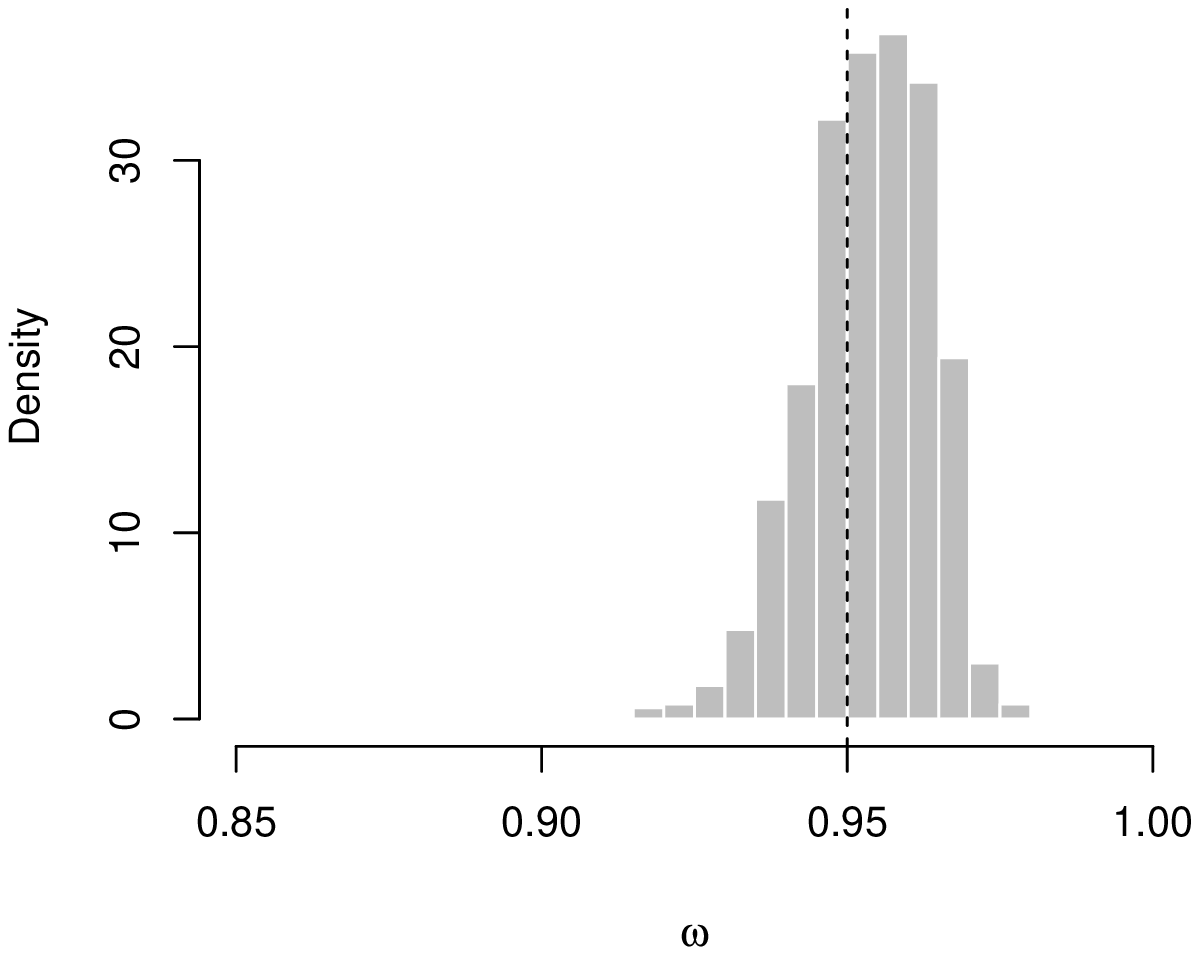}}}
\end{center}
\caption{Posterior distributions for $\omega$ when $n=500$ and $A=5$ for two values of $s_n$.  In each case, $\kappa=0.99$, $\sigma^2=100$, and $\alpha=0.10$.}
\label{fig:wpost}
\end{figure}

As a final example, consider a $n=1000$ dimensional mean vector, with the first 10 entries of $\theta^\star$ equal 10, the next 90 entries equal $A$, and the remaining 900 entries equal zero.  Mean square errors for two Dirichlet--Laplace estimators in \citet{dunson.shrinkage} and our empirical Bayes estimator, based on $\alpha=0.05$, are displayed in Table~\ref{table:sim3}.  Here we consider a range of $A$, from $A=2$ to $A=7$.  For the smaller signals, $A \leq 4$, the Dirichlet--Laplace estimator, with smaller prior weight $n^{-1}$ is the best, but our estimator is better for larger signals, $A > 4$.  The larger weight Dirichlet--Laplace prior estimator is dominated by our empirical Bayes estimator.  

\begin{table}
\begin{center}
\begin{tabular}{cccccccc}
\hline
$A$ & & 2 & 3 & 4 & 5 & 6 & 7 \\
\hline
DL$_{1/n}$ & & {\bf 307} & {\bf 354} & {\bf 271} & 205 & 183 & 169  \\
DL$_{1/2}$ && 368 & 679 & 671 & 374 & 214 & 160 \\
\\
\emph{EBM} & & 320 & 416 & 291 & {\bf 172} & {\bf 137} & {\bf 129} \\
\hline
\end{tabular}
\end{center}
\caption{Mean square errors, based on 100 replications, sampling $X$ of dimension $n=200$.  First two rows are from \citet{dunson.shrinkage}; last row corresponds to the proposed empirical Bayes posterior mean.  Boldface font indicates the column winner.}
\label{table:sim3}
\end{table}

\begin{remark}
\label{re:koenker}
There are a number of existing methods available for this problem besides those included in our comparisons here.  These include the lasso \citep{tibshirani1996}, the Bayesian lasso \citep{park.casella.2008}, the horseshoe prior estimator \citep{carvalho.polson.scott.2010}, the empirical Bayes estimators of \citet{jiang.zhang.2009}, \cite{brown.greenshtein.2009}, and, most recently, \citet{koenker.mizera.2014}.  Some of these methods, including a version of ours, are compared more extensively in \citet{koenker2014}.  Those estimators without minimax guarantees, such as the Koenker--Mizera estimator, can only be motivated by finite-sample simulation studies which, by necessity, are narrowly constructed.  On the other hand, our estimator has the desired minimax property and also has the best overall finite-sample performance among those provably minimax competitors.
\end{remark}

\section{Discussion}
\label{S:discuss}

The paper has considered a classical problem of estimating a sparse high-dimensional normal mean vector, and we have proposed a novel empirical Bayes solution.  Though the stated prior itself may seem overly informative, we show that the prior induces a sort of shrinkage effect, preventing the posterior from tracking the data too closely.  We go on to prove that the empirical Bayes posterior concentrates around $\theta^\star$ at the minimax rate, that its mean is an asymptotic minimax estimator, and that its effective dimension agrees with that of the true sparse mean vector.  

The mathematical device used in our asymptotic analysis is an alternative representation of the empirical Bayes model with a fractional likelihood.  As in \citet{walker.hjort.2001}, this fractional likelihood posterior is a powerful tool, though our concentration results do not follow immediately from theirs.  With this adjustment, the prior changes to a very simple one, which we have called $\tilde\Pi_X$.  The key to success of our empirical Bayes posterior in the asymptotic framework is the particular beta prior on $\omega$, under $\tilde\Pi_X$.  From this prior, and the lower bound derived in Lemma~\ref{lem:denominator}, the minimax rate $\eps_n = s_n \log(n/s_n)$ drops out almost automatically.  As we indicated, to push through the minimax concentration result, we only need the conditional prior on $\theta$, given $\omega$, under $\tilde\Pi_X$, to not interfere with the dynamics induced by the prior on $\omega$.  Intuitively, there should be many priors that would accomplish this.  We showed that an empirical Bayes prior that by centering a Gaussian prior at the observations, under $\tilde\Pi_X$, minimax concentration follows relatively easily.  \citet{castillo.vaart.2012} have similar results, e.g., they make sure the prior for $\theta$ does not interfere by requiring suitably heavy tails.  

In addition to the good large-sample properties, our empirical Bayes procedure is easy to compute, and, in a number of cases, the finite-sample performance of our empirical Bayes posterior mean is considerably better than that of existing methods with comparable large-sample properties (Remark~\ref{re:koenker}).  Since our method admits a full posterior distribution, any other feature, such as the inclusion probabilities displayed in Figure~\ref{fig:inclpr}, useful in the signal detection problem, can be readily calculated.  

A possible extension of the method presented herein is as follows.  Suppose that each $X_i$ and $\theta_i$ are $r$-vectors, where $r=r_n$ possibly depends on $n$.  Collecting some of the variables together in vectors introduce a group structure.  This structure appears in a variety of applications, and this has motivated developments in model selection and estimation with grouped variables \citep[e.g.,][]{yuan.lin.2006}.  \citet{abramovich.grinshtein.2013} prove asymptotic minimaxity of a Bayes method in this grouped setting, and we expect that similar results can be derived based on the ideas presented here.

\section*{Acknowledgements}

The authors are thankful to Professor Roger Koenker who gave some helpful comments on an earlier draft as well as suggestions for improving our Gibbs sampler codes.

\bibliographystyle{apa}
\bibliography{/Users/rgmartin/Dropbox/Research/mybib}

\end{document}